\documentclass[final]{siamltex}
\usepackage[lite]{amsrefs}
\usepackage{amssymb, amsfonts, eucal}

\usepackage{pstricks}
\usepackage{pst-grad}
\usepackage{pst-plot}
\usepackage{pst-node}
\usepackage{multido}

\def\updots{\mathinner{\mkern
1mu\raise 1pt \hbox{.}\mkern 2mu \mkern 2mu \raise
4pt\hbox{.}\mkern 1mu \raise 7pt\vbox {\kern 7 pt\hbox{.}}} }

\newcommand{\labelseparation}{0.5}
\newcommand{\labelsize}{\small}
\newcommand{\thickness}{0.8pt}
\newcommand{\ssize}{0.2}

\newcommand{\linew}{0.8pt}

\newcommand{\plus}{%
\psline[linewidth=\thickness](-\ssize,0)(\ssize,0)
\psline[linewidth=\thickness](0,-\ssize)(0,\ssize)
}
\newcommand{\minus}{%
\psline[linewidth=\thickness](-\ssize,0)(\ssize,0)
}

\newcommand{\bpic}{
\begin{center}
}

\newcommand{\epic}{
\endpspicture
\end{center}
}

\newcommand{\ds}{\displaystyle}

\newcommand{\z}{{\mathbf{z}_n}}
\newcommand{\zzz}{\mathbf{z}}
\newcommand{\p}{{\mathfrak{p}}}
\newcommand{\zz}{{\mathfrak{z}}}
\newcommand{\bz}[1]{{\mathbf{z}_{#1}}}
\newcommand{\tbz}[1]{{\tilde{\mathbf{z}}_{#1}}}

\newcommand{\refclass}[1]{{\mathcal{R}_{\delta}(#1)}}
\newcommand{\remesh}[2]{{\mathcal{R}_{#1}(#2)}}
\newcommand{\text}[1]{\mbox{\rm #1}}

\newcommand{\bx}[1]{{\mathbf{x}_{#1}}}

\newcommand{\SSigma}{\mathcal{S}}
\newcommand{\spl}{\SSigma_{r}(\z)}
\newcommand{\Spl}[2]{\SSigma_{#1}(#2)}
\newcommand{\Sspl}[2]{\widetilde{\SSigma}_{#1}(#2)}
\newcommand{\sspl}{\Sspl{r}{\z}}

\newcommand{\sti}{{i^\star}}
\newcommand{\ints}{s_*}
\newcommand{\sbe}{\tilde{s}}
\newcommand{\sJ}{J_\star}

\newcommand\UU{{\mathbf{u}}}
\newcommand\UUn{\mathbf{u}_n}
\newcommand\Chn{{\mathbf{t}_n}}
\newcommand\Ch{\mathbf{t}}

\newcommand\J {{\mathcal{I}}}

\renewcommand{\L}{\mathbb L}
\newcommand{\Lp}{\L_p}
\newcommand{\Poly}{\Pi}
\newcommand\w{{\omega}}

\newcommand{\C}{\mathbb C}
\newcommand{\N}{\mathbb N}
\newcommand{\R}{\mathbb R}
\newcommand{\W}{\mathbb W}

\newcommand{\scale}{\vartheta}

\def \meas{\mathop{\rm meas}}

\newcommand{\AC}{AC}

\newcommand{\esssup}{\mathop{\rm ess \; sup}}

 \newcommand{\TS}{\widetilde S}

\renewcommand{\phi}{\varphi}

\def\be  {\begin{equation}}
\def\ee  {\end{equation}}
\def\ba  {\begin{eqnarray}}
\def\ea  {\end{eqnarray}}
\def\baa {\begin{eqnarray*}}
\def\eaa {\end{eqnarray*}}

\newcommand{\ineq}[1]{(\ref{#1})}
\newcommand{\ie}{{\em i.e., }}
\newcommand{\eg}{{\em e.g.}}

\newcommand{\thm}[1]{Theorem~\ref{#1}}

\newcommand{\lem}[1]{Lemma~\ref{#1}}
\newcommand{\cor}[1]{Corollary~\ref{#1}}

\newenvironment{comment}[2]
{\bgroup\vspace{7pt}
\begin{tabular}{|p{4.4in}|}
\hline \qquad \bf \footnotesize Comment -- to be deleted in the final version \\
\hline
\quad\sl\footnotesize #1#2} {\\ \hline \end{tabular}
\vspace{7pt}\indent\egroup}

\newcommand{\bc}{\begin{comment}}
\newcommand{\ec}{\end{comment}}

\def\norm#1#2{\ensuremath{\left\|#1\right\|_{#2}}}

 \title{Constrained Spline Smoothing\thanks{The first author was supported in part by NSERC of Canada.}}
\author{K. Kopotun \thanks{Department of Mathematics, University of
Manitoba, Winnipeg, Manitoba, R3T 2N2, Canada
({\tt kopotunk@cc.umanitoba.ca}).}
\and
D. Leviatan\thanks{School of Mathematical Sciences, Raymond and
Beverley Sackler Faculty of Exact Sciences, Tel Aviv University,
Tel Aviv, 69978, Israel ({\tt leviatan@post.tau.ac.il}).}
\and
A. V. Prymak \thanks {Department of Mathematical and Statistical
Sciences, University of Alberta, Edmonton, T6G2G1 AB, Canada ({\tt
prymak@gmail.com}).}}

\date{\today}

\begin{document}

\maketitle

\begin{abstract}
Several results on constrained spline smoothing are obtained. In particular, we establish a general result, showing how one can constructively smooth any monotone or convex piecewise polynomial function (ppf) (or any $q$-monotone ppf, $q\geq 3$, with one additional degree of smoothness)   to be  of
minimal defect while keeping it close to the original function in the $\Lp$-(quasi)norm.
It is well known that approximating a function by ppf's of minimal
defect (splines) avoids introduction of  artifacts which
may be unrelated to the original function, thus it is always
preferable. On the other hand, it is usually easier to construct
constrained ppf's with as little requirements on smoothness as
possible.
Our results allow to obtain shape-preserving splines of
minimal defect with  equidistant or Chebyshev knots.
The validity of the corresponding Jackson-type estimates for
shape-preserving spline approximation is summarized, in particular
we show, that the $\Lp$-estimates, $p\ge1$, can be immediately derived from
the $\L_\infty$-estimates.
\end{abstract}



\pagestyle{myheadings}
\thispagestyle{plain}
\markboth{K. Kopotun, D. Leviatan and A. Prymak}{Constrained Spline Smoothing}

\section{Introduction and the main results}

\subsection{Notation}
Let $\spl$ be the space of all piecewise  polynomial functions (ppf) of degree $r$ (order $r+1$) with the knots  $\z := (z_i)_{i=0}^n$, $-1 =: z_0 <z_1 < \dots <z_{n-1} < z_n := 1$. In other words, we say that $s\in\spl$ if, on each interval
$(z_{i}, z_{i+1})$, $0\leq i\leq n-1$, $s\in \Poly_{r}$, where $\Poly_{r}$ denotes the space of  algebraic polynomials of degree $\leq r$.
Also,  let
$\sspl :=\spl \bigcap \C^{r-1}$ be the
corresponding space of splines of minimal defect (highest smoothness).

As usual, $\Lp(J)$, $0< p\leq \infty$, denotes the space of all
measurable functions $f$ on an interval $J$ such that
$\|f\|_{\Lp(J)} <
\infty$, where $ \|f\|_{\Lp(J)} := \left( \int_J|f(x)|^p\,
dx\right)^{1/p}$ if $p<\infty$, and $ \|f\|_{\L_\infty(J)} :=
\esssup_{x\in J} |f(x)|$. For $\mu\in\N$, the space of all
$\mu$-times continuously differentiable functions on $J$ is denoted
by $\C^\mu(J)$. Also, $\C(J)$ and $\AC(J)$ denote the spaces of all
continuous and locally absolutely continuous  functions,
respectively. (Note that if $f\in\C(J)$, then $\norm{f}{\C(J)} =
\norm{f}{\L_\infty(J)}$.)  The Sobolev space is defined by
$\W_p^r(J) := \left\{ f\in\Lp(J) \mid f^{(r-1)}\in \AC(J) \;
\mbox{\rm and}\; f^{(r)} \in \Lp(J) \right\}$.

 For $k\in\N_0$, define
\[\Delta_h^k(f,x, J):=\left\{
\begin{array}{ll} \ds
\sum_{i=0}^k  {k \choose i}
(-1)^{k-i} f(x-kh/2+ih),&\mbox{\rm if }\, x\pm kh/2  \in J \,,\\
0,&\mbox{\rm otherwise},
\end{array}\right.\]
The $k$-th modulus of smoothness of $f\in \Lp(J)$ is defined by
\[
\omega_k(f,t,J)_p:=\sup_{0<h\le
t}\norm{\Delta_h^k(f,\cdot,J)}{\Lp(J)},
\]
and the Ditzian-Totik moduli of
smoothness is
\[
\omega_k^\phi(f,t)_p:=\sup_{0<h\le
t}\norm{\Delta^k_{h\phi(\cdot)}(f,\cdot)}{\Lp[-1,1]} \,,
\]
where  $\phi(x):=\sqrt{1-x^2}$.

The set of $q$-monotone functions on $J$ is denoted by
$\Delta^q(J)$. Recall that $f\in\Delta^q(J)$ if the divided
differences $[f; t_0, \dots, t_q]$ of order $q$ of $f$ are
nonnegative for all choices of $(q + 1)$ distinct points
$t_0,\dots,t_q$ in $J$. If f is continuous, then $f\in\Delta^q(J)$
if $\Delta^q_h(f,x,J)
\geq 0$ for all $x\in J$ and $h>0$. It is well known that, for $q\ge 2$ and an open interval $J$,
$f\in\Delta^q(J)$ if and only if $f^{(q-2)}$ exists and is convex on
 $J$.

The error of   unconstrained approximation of $f$ from a set $U$ is
denoted by
\[
E(f,U)_p:=\inf_{u\in U}\norm{f-u}{\Lp(J)},
\]
and the error of $q$-monotone approximation (\ie approximation of $f$ by $q$-monotone elements of $U$) is
\[
E^{(q)}(f,U)_p:=E(f,U\cap\Delta^q(J))_p.
\]

 Throughout this paper, we also use the notation $\norm{f}{p} := \norm{f}{\Lp[-1,1]}$, $\Delta^q := \Delta^q[-1,1]$, $|J| := \meas(J)$, and
 \[
\omega_k(f,J)_p:=\omega_k(f,|J|,J)_p
\quad\text{and}\quad
\omega_k(f,t)_p:=\omega_k(f,t,[-1,1])_p.
\]

Given a partition $\z := (z_i)_{i=0}^n$, $-1 =: z_0 <z_1 < \dots
<z_{n-1} < z_n := 1$, we say that partition $\tbz{m} = (\tilde
z_i)_{i=0}^m$ is a {\em $\delta$-remesh} of $\z$ if, for each $0\leq
j\leq n-1$,
\[
\max\left\{ \tilde z_{i+1}-\tilde z_i \; \mid  \; [\tilde z_i, \tilde z_{i+1}] \cap (z_j, z_{j+1}) \ne \emptyset \right\} \leq \delta
\min_{\nu=j-1, j, j+1} \left|z_{\nu+1}-z_\nu  \right|
\]
with $z_{-1}$ and $z_{n+1}$ defined to be (in this definition only!)
$-\infty$ and $+\infty$, respectively. In other words,  the largest
interval $[\tilde z_i, \tilde z_{i+1}]$ intersecting $(z_j,
z_{j+1})$ should have the length at most $\delta$ times the length
of $[z_j, z_{j+1}]$ or the lengths of (one or two) intervals
adjacent to $[z_j, z_{j+1}]$ whichever is smaller. The class of all
$\delta$-remeshes of $\z$ is denoted by $\refclass{\z}$.

Clearly, $m$, $n$ and $\delta$ are not independent. The smallest  $m$ such that there is $\tbz{m}$ in $\refclass{\z}$  is determined not only by
$n$ and $\delta$ but also by the scale of the partition $\z$ (see \ineq{ratio}).

It is well known that it is easier and less costly to obtain a good
piecewise approximation to a function $f$ than to assure, at the
same time, the maximum smoothness of the approximating ppf's. On the
other hand, replacing $f$ by a ppf with minimal defect (spline)
avoids introduction of  artifacts which are unrelated to
the original function, and may effect and complicate the problem one
deals with. Thus, given two partitions $\z$ and $\tbz{m}$ with the
only requirement that $\tbz{m}$ is somewhat denser than $\z$ (more
precisely, $\tbz{m}$ is an {\em arbitrary} $\delta$-remesh of $\z$),
we show in this paper, how to smooth a general ppf on $\z$, to a
spline of maximum smoothness on $\tbz{m}$, while staying close to
the original ppf and preserving  its shape characteristics.
We prove that such a $\delta>0$ exists, that splines on an arbitrary
$\delta$-remesh of $\z$ may be constructed to satisfy the above
requirements. In particular, as  an illustration of these general
results, we consider two special kinds of partitions of $[-1,1]$
which are important in applications:
 the {\em uniform} partition $\UU_n :=(-1+2j/n)_{j=0}^n$
and the {\em Chebyshev} partition $\Ch_n :=(-\cos(j\pi/n))_{j=0}^n$. 

The following are some properties of classes $\refclass{\z}$.
\begin{itemize}
\item For any partition $\z$ and $0< \delta_1 \leq \delta_2$,
$\remesh{\delta_1}{\z} \subset \remesh{\delta_2}{\z}$.
\item If a partition $\zzz_{kn}^*$ is a refinement of the partition $\z$ obtained by subdividing each interval
in $\z$ into $k$ equal subintervals, and $\tilde\zzz_m \in \remesh{\delta}{\z}$, then
$\tilde\zzz_m \in \remesh{k\delta}{\zzz_{kn}^*}$.
\item If $0<\delta<1$ and $\tilde\zzz_m \in \remesh{\delta}{\z}$, then any interval
in $\tilde\zzz_m$ is contained in the union of at most two intervals in $\z$.
\item For any $m\geq n/\delta$, $\UU_{m} \in \refclass{\UU_n}$.
\item For any $m\geq \max\{25/\delta, 1\} n$, $\Ch_{m} \in \refclass{\Ch_n}$.
\end{itemize}

\subsection{Constrained smoothing}

Let $\z:=\{z_0, \dots, z_n | -1=:z_0<z_1<\cdots<z_n:=1\}$  be a partition of
$[-1,1]$, and extend the notation by setting 
$z_j := z_0$, $j<0$, and $z_j := z_n$, $j>n$.
We denote the scale of the partition $\z$ by
\be\label{ratio}
 \scale(\z) := \max_{0\le j\le n-1}\frac{|J_{j\pm1}|}{|J_j|} \, ,
\ee
where $J_j := [z_{j},z_{j+1}]$. We also denote $\J_j := [(z_{j-1}+z_{j})/2,(z_j+z_{j+1})/2]$.

The following is our main result on constrained spline smoothing.

\begin{theorem} \label{smoothing}
Let $q\in\N$, $r\in\N$, and let $\z = (z_i)_{i=0}^n$ be a partition
of $[-1,1]$. There is a constant $\delta=\delta(q,r)$ such that for
each $s\in
\Spl{q+r}{\z}\cap \Delta^q$  such that
\begin{equation}
  \label{cq-1}
  s\in \C^{q-1}[-1,1],
\end{equation}
and  any  $\tbz{m}\in\refclass{\z}$ (\ie $\tbz{m}$ is a
$\delta$-remesh of $\z$), there exists a spline   $\tilde s
\in
\Sspl{q+r}{\tbz{m}}
\cap \Delta^q$  satisfying
\[
\norm{s-\tilde s}{\Lp(\J_j)}
\le c(p,q,r) \omega_{q+r+1}(s,\J_j)_p\, , \; 0\leq j\leq n \,,
\]
for all $0<p\le\infty$. Moreover, the construction of $\tilde s$ does not
depend on $p$.
\end{theorem}

The proof of Theorem 1.1 (as well as the proof of Theorem 1.2 below)
is postponed until Section 3.

Recall  that,  for any $q\geq 2$, $\Delta^q\subset \C^{q-2}(-1,1)$,
\ie any $q$-monotone function is in $\C^{q-2}$. Hence, the
smoothness provided by shape itself does not guarantee the
applicability of the above result, one needs to assume/gain one
additional smoothness degree.

It turns out that for $q=1$ and $q=2$ the gain of this additional degree
of smoothness is not difficult (see section~\ref{cms}), and so we get the following stronger result in the
case for $q\leq 2$.

\begin{theorem}
\label{c12}
Let $q=1$ or $2$, $r\in\N$, and let $\z = (z_i)_{i=0}^n$ be a partition of $[-1,1]$. There is a
constant $\delta=\delta(r)$ such that for each $s\in
\Spl{q+r}{\z}\cap \Delta^q$
and any $\tbz{m}\in\refclass{\z}$,
there exists a spline   $\tilde s \in \Sspl{q+r}{\tbz{m}}
\cap \Delta^q$  satisfying
\[
\norm{s-\tilde s}{\Lp(\J_j)}
\le c(p,r) \omega_{q+r+1}(s,\J_j)_p\, ,\; 0\leq j\leq n\,,
\]
for all $0<p\le\infty$.
\end{theorem}

Note that an analog of \thm{c12} is also valid (and is actually simpler) in the case $r=0$ (see \lem{lemcms} and \cor{corms}).

In  case the partitions $\z$ and $\tbz{m}$ are either both uniform or both Chebyshev, \thm{smoothing} can be restated as follows.

\begin{corollary} \label{c13}
Let $q\in\N$, $r\in\N$, and let $\z = (z_i)_{i=0}^n$ denote either
$\UU_n$ or $\Ch_n$. There is a constant $m_0=m_0(q,r)$ such that for
each $s\in
\Spl{q+r}{\z}\cap  \Delta^q  \cap \C^{q-1}[-1,1]$,
and any $m\geq m_0 n$,
there exists a spline   $\tilde s \in \Sspl{q+r}{\bz{m}} \cap \Delta^q$, where
$\bz{m}$ is either $\UU_m$ or $\Ch_m$, respectively,
 satisfying
\[
\norm{s-\tilde s}{\Lp(\J_j)}
\le c(p,q,r) \omega_{q+r+1}(s,\J_j)_p\, ,\; 0\leq j\leq n\,,
\]
for all $0<p\le\infty$.
\end{corollary}

We remark that in view of \thm{c12}, in the case $q=1$ and $q=2$, the condition that $s$ is in $\C^{q-1}[-1,1]$ in \cor{c13} can be removed.

Finally, we remark that it is still an open question if the condition \ineq{cq-1} in the statement of \thm{smoothing} can be removed if $q\geq 3$.

\section{Applications: Jackson type estimates}

\subsection{{\bf Monotone and convex spline approximation: $p=\infty$}}

The following theorem is rather well known. Its positive part  follows from Whitney's inequality ($q=1$ and $1\leq k+\nu \leq 2$),  \cite[Lemma 2]{LeShejat98} ($q=1$, $\nu \geq 1$), \cite[Corollary 2.4]{LeShejat03} ($q=2$, $\nu\geq 2$), and \cite{Ko94} ($q=2$, $2\leq k+\nu\leq 3$).   The negative part follows from \cite{Shv} and \cite[p. 141]{She}.

\begin{theorem}[$p=\infty$] \label{thminf}
  Let $q=1$ or $2$, and $k, \nu\in\N_0$  be such that either $\nu\geq q$ or $q\leq k+\nu\leq q+1$ (see Fig.~1 and Fig.~2 below).
  Then, for every
 $f\in\Delta^q\cap \C^\nu[-1,1]$, $n\in\N$,  and any partition $\z = (z_i)_{i=0}^n$  of $[-1,1]$,
 there exists  $s\in \Spl{k+\nu-1}{\z}\cap \Delta^q$ such that
\[
\norm{f-  s}{\L_\infty (J_j)} \le c(k,\nu, \scale(\z)) |J_j|^\nu \omega_{k}(f^{(\nu)}, J_j)_\infty \,, \; 0\leq j\leq n-1\,.
\]
Moreover, this estimate is no longer true in general for $k$ and
$\nu$ which do not satisfy the above conditions. (This means that for
each partition $\z$ and any constant $c(k,\nu, \scale(\z))$, there exists a
 function $f\in\Delta^q\cap
\C^\nu[-1,1]$, such that the above estimate is invalid for any $s\in
\Spl{k+\nu-1}{\z}\cap
\Delta^q$.)
\end{theorem}

We note that the case $q=2$, $k+\nu=1$ is excluded from the statement of \thm{thminf} (as well as statements
of \thm{thmp} and Corollaries~\ref{thminfsmooth}, \ref{corol23}, \ref{thmpsmooth} and \ref{corollp}). Indeed, the only convex piecewise
 constant ppfs are constant functions on $[-1,1]$.

Using \thm{c12} we can now obtain the following consequence of this result for monotone and convex approximation by  splines of any smoothness.

\begin{corollary}[$p=\infty$] \label{thminfsmooth}
  Let $q=1$ or $2$, and $k, \nu\in\N_0$  be such that either $\nu\geq q$ or $q\leq k+\nu\leq q+1$.
  Then, for any $r\geq k+\nu-1$ and partition $\z = (z_i)_{i=0}^n$  of $[-1,1]$, there exists  a constant $\delta=\delta(r)$ such that for any $\tbz{m}\in\refclass{\z}$ and every $f\in\Delta^q\cap \C^\nu[-1,1]$,
there exists a spline   $\tilde s \in \Sspl{r}{\tbz{m}} \cap \Delta^q$  satisfying
 \[
\norm{f-  \tilde s}{\L_\infty (\J_j)} \le c(k,\nu,r, \scale(\z)) |J_j|^\nu \omega_{k}(f^{(\nu)}, [z_{j-1},z_{j+1}])_\infty
\]
for all $0\leq j\leq n$.
Moreover, this estimate is no longer true in general for $k$ and $\nu$ which do not satisfy the above conditions.
\end{corollary}

Taking into account \cor{c13} and the fact that $|J_j| = n^{-1}$ (if $\z = \UUn$) and $|J_j| \sim \varphi(x) n^{-1} + n^{-2}$, $x\in J_j$ (if $\z=\Chn$), this, in turn, immediately implies:

\begin{corollary}  \label{corol23}
Let $q=1$ or $2$, $k, \nu\in\N_0$  be such that either $\nu\geq q$ or $q\leq k+\nu\leq q+1$.
 Then, for every  $f\in\Delta^q\cap \C^\nu[-1,1]$, $n\in\N$ and $r\geq k+\nu-1$, we have
\be \label{e1}
  E^{(q)}(f,\Sspl{r}{\UUn})_\infty\le c(k,\nu,r)
  n^{-\nu}\omega_k(f^{(\nu)},n^{-1})_\infty
\ee
and
\be \label{e2}
  E^{(q)}(f,\Sspl{r}{\Chn})_\infty\le c(k,\nu,r)
  n^{-\nu}\omega_k^\phi(f^{(\nu)},n^{-1})_\infty \, .
\ee
Moreover, these estimates are no longer true in general for $k$ and $\nu$ which do not satisfy the above conditions.
\end{corollary}

For the reader's convenience we describe the above results using
arrays in Figures~1 and 2. In these figures as well as Figures~3 and
4 in the case $1\leq p<\infty$ below (with obvious modifications),
the symbols ``$-$"
  and ``$+$" have the following meaning.

\begin{itemize}
\item The symbol ``$+$" in the position $(k,\nu)$ means that
inequalities \ineq{e1} and \ineq{e2} are valid for all $f\in\Delta^q\cap \C^\nu[-1,1]$.
\item The symbol ``$-$" in the position $(k,\nu)$  means that
inequalities \ineq{e1} and \ineq{e2} are not true in general, \ie there are functions
 $f\in\Delta^q\cap \C^\nu[-1,1]$ for which these inequalities fail.
\end{itemize}

\begin{minipage}[t]{2.5in}
\bpic
\psset{unit=7mm}
\pspicture(0,-3)(6,7)
\psgrid[subgriddiv=1,griddots=6,gridlabels=0,gridcolor=gray](0,0)(0,0)(5,5)

\renewcommand{\updots}{\pscircle*[linewidth=\linew](0,0){0.03}
\pscircle*[linewidth=\linew](-0.2,-0.2){0.03}
\pscircle*[linewidth=\linew](-0.4,-0.4){0.03}
}

\uput{\labelseparation}[180](0,0){\labelsize $0$}
\uput{\labelseparation}[180](0,1){\labelsize $1$}
\uput{\labelseparation}[180](0,2){\labelsize $2$}
\uput{\labelseparation}[180](0,3){\labelsize $3$}
\uput{\labelseparation}[180](0,4){\labelsize $4$}
\uput{\labelseparation}[180](0,5){\labelsize $5$}
\uput{\labelseparation}[180](0,6){\labelsize $\mathbf \nu$}

\uput{\labelseparation}[-90](0,0){\labelsize $0$}
\uput{\labelseparation}[-90](1,0){\labelsize $1$}
\uput{\labelseparation}[-90](2,0){\labelsize $2$}
\uput{\labelseparation}[-90](3,0){\labelsize $3$}
\uput{\labelseparation}[-90](4,0){\labelsize $4$}
\uput{\labelseparation}[-90](5,0){\labelsize $5$}
\uput{\labelseparation}[-90](6,0){\labelsize $\mathbf k$}

\rput(2.8,-2){\scriptsize \sc Fig. 1. $q=1$, $p=\infty$}

\multirput*(3,0)(1,0){3}{\minus}
\multirput*(0,2)(1,0){6}{\plus}
\multirput*(0,3)(1,0){6}{\plus}
\multirput*(0,4)(1,0){6}{\plus}
\multirput*(0,5)(1,0){6}{\plus}

\multirput*(1,0)(1,0){2}{\plus}
\multirput*(0,1)(1,0){6}{\plus}

\multirput*(6,0)(0,1){6}{$\cdots$}
\multirput*(0,6)(1,0){6}{$\vdots$}
\rput(6,6){$\updots$}

\epic

\end{minipage}
\begin{minipage}[t]{2.5in}

\bpic
\psset{unit=7mm}
\pspicture(0,-3)(6,6)
\psgrid[subgriddiv=1,griddots=6,gridlabels=0,gridcolor=gray](0,0)(0,0)(5,5)

\renewcommand{\updots}{\pscircle*[linewidth=\linew](0,0){0.03}
\pscircle*[linewidth=\linew](-0.2,-0.2){0.03}
\pscircle*[linewidth=\linew](-0.4,-0.4){0.03}
}

\uput{\labelseparation}[180](0,0){\labelsize $0$}
\uput{\labelseparation}[180](0,1){\labelsize $1$}
\uput{\labelseparation}[180](0,2){\labelsize $2$}
\uput{\labelseparation}[180](0,3){\labelsize $3$}
\uput{\labelseparation}[180](0,4){\labelsize $4$}
\uput{\labelseparation}[180](0,5){\labelsize $5$}
\uput{\labelseparation}[180](0,6){\labelsize $\mathbf \nu$}

\uput{\labelseparation}[-90](0,0){\labelsize $0$}
\uput{\labelseparation}[-90](1,0){\labelsize $1$}
\uput{\labelseparation}[-90](2,0){\labelsize $2$}
\uput{\labelseparation}[-90](3,0){\labelsize $3$}
\uput{\labelseparation}[-90](4,0){\labelsize $4$}
\uput{\labelseparation}[-90](5,0){\labelsize $5$}
\uput{\labelseparation}[-90](6,0){\labelsize $\mathbf k$}

\rput(2.8,-2){\scriptsize \sc Fig. 2. $q=2$, $p=\infty$}

\multirput*(4,0)(1,0){2}{\minus}
\multirput*(3,1)(1,0){3}{\minus}
\multirput*(0,2)(1,0){6}{\plus}
\multirput*(0,3)(1,0){6}{\plus}
\multirput*(0,4)(1,0){6}{\plus}
\multirput*(0,5)(1,0){6}{\plus}

\multirput*(2,0)(1,0){2}{\plus}
\multirput*(1,1)(1,0){2}{\plus}

\multirput*(6,0)(0,1){6}{$\cdots$}
\multirput*(0,6)(1,0){6}{$\vdots$}
\rput(6,6){$\updots$}

\epic

\end{minipage}

We remark that the estimates in \cor{corol23} (and arrays in Figures 1 and 2) are well known for algebraic polynomials as well as splines of
low smoothness. For  splines of minimal defect, only some special cases when $\z=\UUn$ were known (see, \eg, \cite[Theorem 1]{DeV-77}, \cite[Theorem 2.1]{LeMh},
\cite[Theorem 1]{Be-81}, \cite[Corollaries 5.3 and 5.4]{KoMC}).

\subsection{{\bf Monotone and convex spline approximation: $1\leq p< \infty$}}

\mbox{}

\begin{theorem}[$1\leq p<\infty$] \label{thmp}
  Let $q=1$ or $2$, $1\leq p<\infty$, and $k, \nu\in\N_0$  be such that either $\nu\geq q+1$ or $q\leq k+\nu\leq q+1$.
  Then, for every
 $f\in\Delta^q\cap \W^\nu_p[-1,1]$, $n\in\N$,  and any partition $\z = (z_i)_{i=0}^n$  of $[-1,1]$,
 there exists  $s\in \Spl{k+\nu-1}{\z}\cap \Delta^q$ such that
\[
\norm{f-  s}{\Lp(J_j)} \le c(k,\nu, \scale(\z)) |J_j|^\nu \omega_{k}(f^{(\nu)}, J_j)_p \,, \; 0\leq j\leq n-1\,.
\]
Moreover, this estimate is no longer true in general for $k$ and $\nu$ which do not satisfy the above conditions.
\end{theorem}

The positive part of \thm{thmp} for $\nu\geq1$ follows from
\thm{thminf} and the well-known inequality
\[
\omega_{r+1}(f,J)_\infty \le c |J|^{1-1/p} \, \omega_r(f',
J)_p,
\]
 where $J$ is a closed interval, $f\in \W^1_p(J)$, $1\le p
<\infty$ and $r\in\N_0$. For $\nu=0$, the case $(q,k)=(1,1)$ is
straightforward, $(q,k)=(2,3)$ is~\cite[Theorem~1.2]{DHL} and the
cases $(q,k)=(1,2)$ and $(2,2)$ are established following the proof
of~\cite[Theorem~1.2]{DHL} using piecewise linear ppfs which interpolate
$f$ at the knots instead of piecewise quadratic functions in \cite[Lemma~2.2]{DHL}. Negative part
of~\thm{thmp} is a consequence of
\cite[Theorem 1]{Ko95}.

As in the case $p=\infty$,   \thm{c12} yields the following stronger result for monotone and convex approximation by  splines of any smoothness.

\begin{corollary}[$1\leq p<\infty$] \label{thmpsmooth}
  Let $q=1$ or $2$, $1\leq p<\infty$, and $k, \nu\in\N_0$  be such that either $\nu\geq q+1$ or $q\leq k+\nu\leq q+1$.
  Then, for any $r\geq k+\nu-1$ and partition $\z = (z_i)_{i=0}^n$  of $[-1,1]$, there exists  a constant $\delta=\delta(r)$ such that for any $\tbz{m}\in\refclass{\z}$ and every $f\in\Delta^q\cap  \W^\nu_p[-1,1]$,
there exists a spline   $\tilde s \in \Sspl{r}{\tbz{m}} \cap \Delta^q$  satisfying
 \[
\norm{f-  \tilde s}{\Lp(\J_j)} \le c(k,\nu,r, \scale(\z)) |J_j|^\nu \omega_{k}(f^{(\nu)}, [z_{j-1},z_{j+1}])_p
\]
for all $0\leq j\leq n$.
Moreover, this estimate is no longer true in general for $k$ and $\nu$ which do not satisfy the above conditions.
\end{corollary}

Recalling that for $f\in\Lp[-1,1]$, $1\leq p <\infty$, and $k,\mu\in\N$, the following estimates are true
(see, \eg,  \cite{pp-book, dly}):
\[
\sum_{j=0}^{n-\mu-1} \w_k\left(f, \ds \cup_{i=j}^{j+\mu} J_i\right)^p_p \le
\left\{
\begin{array}{ll}
c(k,\mu) \, \w_k(f,n^{-1})^p_p\,, & \mbox{\rm if $\z=\UUn$} \,, \\
 c(k, \mu) \, \w_k^\varphi(f,n^{-1})^p_p \,, & \mbox{\rm if $\z=\Chn$} \, ,
\end{array}
\right.
\]
 we get the following consequence of \cor{thmpsmooth}.

\begin{corollary}  \label{corollp}
Let $q=1$ or $2$, $1\leq p<\infty$, and $k, \nu\in\N_0$  be such that either $\nu\geq q+1$ or $q\leq k+\nu\leq q+1$.
 Then, for any  $f\in\Delta^q\cap  \W^\nu_p[-1,1]$, $n\in\N$ and $r\geq k+\nu-1$, we have
\be \label{e1p}
  E^{(q)}(f,\Sspl{r}{\UUn})_p\le c(k,\nu,r)
  n^{-\nu}\omega_k(f^{(\nu)},n^{-1})_p
\ee
and
\be \label{e2p}
  E^{(q)}(f,\Sspl{r}{\Chn})_p\le c(k,\nu,r)
  n^{-\nu}\omega_k^\phi(f^{(\nu)},n^{-1})_p \, .
\ee
Moreover, these estimates are no longer true in general for $k$ and $\nu$ which do not satisfy the above conditions.
\end{corollary}

\begin{minipage}[t]{2.5in}
\bpic
\psset{unit=7mm}
\pspicture(0,-3)(6,7)
\psgrid[subgriddiv=1,griddots=6,gridlabels=0,gridcolor=gray](0,0)(0,0)(5,5)

\renewcommand{\updots}{\pscircle*[linewidth=\linew](0,0){0.03}
\pscircle*[linewidth=\linew](-0.2,-0.2){0.03}
\pscircle*[linewidth=\linew](-0.4,-0.4){0.03}
}

\uput{\labelseparation}[180](0,0){\labelsize $0$}
\uput{\labelseparation}[180](0,1){\labelsize $1$}
\uput{\labelseparation}[180](0,2){\labelsize $2$}
\uput{\labelseparation}[180](0,3){\labelsize $3$}
\uput{\labelseparation}[180](0,4){\labelsize $4$}
\uput{\labelseparation}[180](0,5){\labelsize $5$}
\uput{\labelseparation}[180](0,6){\labelsize $\mathbf \nu$}

\uput{\labelseparation}[-90](0,0){\labelsize $0$}
\uput{\labelseparation}[-90](1,0){\labelsize $1$}
\uput{\labelseparation}[-90](2,0){\labelsize $2$}
\uput{\labelseparation}[-90](3,0){\labelsize $3$}
\uput{\labelseparation}[-90](4,0){\labelsize $4$}
\uput{\labelseparation}[-90](5,0){\labelsize $5$}
\uput{\labelseparation}[-90](6,0){\labelsize $\mathbf k$}

\rput(2.8,-2){\scriptsize \sc Fig. 3. $q=1$, $1\leq p<\infty$}

\multirput*(3,0)(1,0){3}{\minus}
\multirput*(2,1)(1,0){4}{\minus}
\multirput*(0,2)(1,0){6}{\plus}
\multirput*(0,3)(1,0){6}{\plus}
\multirput*(0,4)(1,0){6}{\plus}
\multirput*(0,5)(1,0){6}{\plus}

\multirput*(1,0)(1,0){2}{\plus}
\multirput*(0,1)(1,0){2}{\plus}

\multirput*(6,0)(0,1){6}{$\cdots$}
\multirput*(0,6)(1,0){6}{$\vdots$}
\rput(6,6){$\updots$}

\epic

\end{minipage}
\begin{minipage}[t]{2.5in}

\bpic
\psset{unit=7mm}
\pspicture(0,-3)(6,6)
\psgrid[subgriddiv=1,griddots=6,gridlabels=0,gridcolor=gray](0,0)(0,0)(5,5)

\renewcommand{\updots}{\pscircle*[linewidth=\linew](0,0){0.03}
\pscircle*[linewidth=\linew](-0.2,-0.2){0.03}
\pscircle*[linewidth=\linew](-0.4,-0.4){0.03}
}

\uput{\labelseparation}[180](0,0){\labelsize $0$}
\uput{\labelseparation}[180](0,1){\labelsize $1$}
\uput{\labelseparation}[180](0,2){\labelsize $2$}
\uput{\labelseparation}[180](0,3){\labelsize $3$}
\uput{\labelseparation}[180](0,4){\labelsize $4$}
\uput{\labelseparation}[180](0,5){\labelsize $5$}
\uput{\labelseparation}[180](0,6){\labelsize $\mathbf \nu$}

\uput{\labelseparation}[-90](0,0){\labelsize $0$}
\uput{\labelseparation}[-90](1,0){\labelsize $1$}
\uput{\labelseparation}[-90](2,0){\labelsize $2$}
\uput{\labelseparation}[-90](3,0){\labelsize $3$}
\uput{\labelseparation}[-90](4,0){\labelsize $4$}
\uput{\labelseparation}[-90](5,0){\labelsize $5$}
\uput{\labelseparation}[-90](6,0){\labelsize $\mathbf k$}

\rput(2.8,-2){\scriptsize \sc Fig. 4. $q=2$, $1\leq p<\infty$}

\multirput*(4,0)(1,0){2}{\minus}
\multirput*(3,1)(1,0){3}{\minus}
\multirput*(2,2)(1,0){4}{\minus}
\multirput*(0,2)(1,0){2}{\plus}
\multirput*(0,3)(1,0){6}{\plus}
\multirput*(0,4)(1,0){6}{\plus}
\multirput*(0,5)(1,0){6}{\plus}

\multirput*(2,0)(1,0){2}{\plus}
\multirput*(1,1)(1,0){2}{\plus}

\multirput*(6,0)(0,1){6}{$\cdots$}
\multirput*(0,6)(1,0){6}{$\vdots$}
\rput(6,6){$\updots$}

\epic

\end{minipage}

\section{Further results and proofs}
 
 \mbox{}

Everywhere in this section, if $p$ is a polynomial piece of a spline $s$ on an interval $J$, \ie 
$p:=s|_J$, then $p(x)$ for $x\not\in J$ is the polynomial extension of $s|_J$.

\subsection{Constrained spline smoothing: proof of \thm{smoothing}}

 \mbox{}

\begin{lemma}[\mbox{\cite[Lemma 2.4 and Corollary 2.5]{KoLePr}}]\label{klp}
Let $s\in\spl$, and suppose that the polynomials $p_j := s |_{[z_j,
z_{j+1}]}$, $0\leq j\leq n-1$. Then, for every $1\leq j\leq n-1$ and
all $0<p\leq \infty$,
\[
|J_j|^{k+1/p} |p_j^{(k)}(z_j)-p_{j-1}^{(k)}(z_{j})|\le c (p, r,\scale(\z))
 \w_{r+1}(s,[z_{j-1},z_{j+1}])_p \, ,
\]
for $ 0\leq k\leq r$, and
\[
\|p_j-p_{j-1}\|_{\Lp[z_{j-1},z_{j+1}]} \leq c(p,r,\scale(\z))  \w_{r+1}(s,[z_{j-1},z_{j+1}])_p \,  . \]

\end{lemma}

\begin{lemma}[Markov's Inequality] \label{markov}
For any polynomial $\p\in\Poly_r$,
\[
\norm{\p'}{\C[a,b]} \leq \frac{2r^2}{b-a} \norm{\p}{\C[a,b]} \, .
\]
\end{lemma}

\begin{lemma}\label{positive_piece}
Let $r\in\N$, $S\in\Spl{r}{\{-1,0, 1\}}$ (\ie $S$ is a ppf of degree $\le r$ on
$[-1,1]$ with the only breakpoint at $0$), and let $p_1$ and $p_2$ be the
  polynomial pieces of $S$: $p_1:=S|_{[-1,0)}$,
$p_2:=S|_{(0,1]}$. If $S$ is non-negative on $[-1,1]$, then there
exists an interval $I$, $I\subset[-1,0]$ or $I\subset[0,1]$, such that
$|I|\ge \ds\frac1{4 r^2}$, and
\[
S(x) \ge c_1(r) \norm{p_1-p_2}{\infty}\,,  \quad \text{for all } x\in
I.
\]
\end{lemma}
\begin{proof} Let $p\in\Poly_r$ be nonnegative on $[0,1]$
(the same argument applies for $[-1,0]$). Then,
there exists an interval
$I\subset [0,1]$, $|I|\ge\ds\frac1{4 r^2}$ such that
\[
p(x)\ge \frac12\norm{p}{\C[0,1]},
\quad \text{for all } x\in I.
\]
Indeed, let $x^*\in[0,1]$ be a point satisfying $p(x^*)=\norm{p}{\C[0,1]}$.
By \lem{markov}, $\norm{p'}{\C[0,1]}\le
2 r^2\norm{p}{\C[0,1]}$, and so
\[
p(x)\ge p(x^*) - 2 r^2\norm{p}{\C[0,1]}|x-x^*|,
\quad x\in[0,1].
\]
Hence, if $x\in[0,1]$ and $|x-x^*|\le\ds\frac1{4 r^2}$, then
$p(x)\ge\ds\frac12\norm{p}{\C[0,1]}$, as required.

Now suppose without loss of generality that $\norm{p_2}{\infty}\ge
\norm{p_1}{\infty}$. Then
\[
\norm{p_1-p_2}{\infty}\le\norm{p_1}{\infty}+\norm{p_2}{\infty}
\le 2 \norm{p_2}{\infty} \le c(r) \norm{p_2}{\C[0,1]}\,,
\]
 and the
above observation can be used with $p := p_2$.
\end{proof}

The following lemma   follows from Beatson~\cite[Lemma~3.2]{Be}.

\begin{lemma} \label{beatson}
Let $r\in\N$, $d :=2r^2$, and $\p_1,\p_2 \in\Poly_r$.  For any knot
sequence $\bx{d} := (x_i)_{i=0}^d$, $a=x_0<x_1<\dots<x_d=b$,
 there
exists a spline $s\in\Sspl{r}{\bx{d}}$   such that
\begin{itemize}
\item[(i)]
$s(x)$ is a number between   $\p_1(x)$ and
$\p_2(x)$
 for  $x\in[a,b]$,
 \item[(ii)]
$s\equiv   \p_1$ on $(-\infty,a]$, and
$s \equiv \p_2$ on $[b,\infty)$.
\end{itemize}
\end{lemma}

\begin{lemma}\label{glue}
Let $r\in\N$,  $S\in\Spl{r}{\{-1,0, 1\}}$ (\ie $S$ is a ppf of degree $\le r$ on
$[-1,1]$ with the only breakpoint at $0$), and assume that $S$ is non-negative on $[-1,1]$.
Suppose that $d :=2r^2$, and a knot sequence $\bz{2d+1}  := (z_i)_{i=0}^{2d+1}$, $-1 < z_0<z_1<\dots<z_{2d+1} < 1$ is such that
$z_d \leq 0 \leq z_{d+1}$ (in other words, there are $d+1$ knots to the left and to the right of zero).
Then, there exists  $\TS\in\Sspl{r}{\bz{2d+1}}$ satisfying
\begin{itemize}
\item[(i)] $\TS \geq 0$ on $[-1,1]$,
\item[(ii)] $\TS(x)  =  S(x)$ for all   $x\not\in[z_0, z_{2d+1}]$,
\item[(iii)] $\norm{S-\TS}{\infty}\le 2
\norm{p_1-p_2}{\infty}$,
where   $p_1:=S|_{[-1,0)}$,
$p_2:=S|_{(0,1]}$ are the polynomial pieces of $S$.
 \end{itemize}
\end{lemma}

\begin{proof} Let $\tilde p (x) := p_1(x) + \norm{p_1-p_2}{\infty}$. Note that
$\tilde p(x) \geq p_1(x) \geq 0$ for $x\in [-1,0]$, and
$\tilde p(x) \geq p_2 (x) \geq 0$ for $x\in [0,1]$.

We now use \lem{beatson} twice.
First, we ``glue" $p_1$ with $\tilde p$   using \lem{beatson} with $\p_1  = p_1$, $\p_2 =\tilde p$,
$\bx{d}  = (z_i)_{i=0}^d$, and $[a,b] = [z_0, z_d]$. Second, we ``glue" $\tilde p$ with $p_2$
using \lem{beatson} with $\p_1 = \tilde p$, $\p_2 = p_2$, $\bx{d} = (z_i)_{i=d+1}^{2d+1}$, and
$[a,b] = [z_{d+1}, z_{2d+1}]$. As a result, we get a spline $\TS \in\Sspl{r}{\bz{2d+1}}$ such that
$\tilde p \geq \TS  \geq p_1  \geq 0$ on $[z_0, z_d]$, $\tilde p \geq \TS  \geq p_2  \geq 0$ on $[z_{d+1}, z_{2d+1}]$,
$\TS(x)  =  S(x) \geq 0 $ for    $x\not\in[z_0, z_{2d+1}]$, and $\TS(x)  =  \tilde p(x) \geq 0$ for   $x \in[z_d, z_{d+1}]$.
Clearly,
\[
 \norm{S-\TS}{\infty} \leq \max_{i=1,2} \norm{\tilde p - p_i}{\infty} \leq 2 \norm{p_1-p_2}{\infty} \, ,
\]
and the proof is complete.
\end{proof}

The following lemma is the main tool for constrained smoothing.

\begin{lemma}\label{lemma3.6}
Let $q\in\N$, $r\in\N$. There exists a constant $\delta(q,r)>0$ such
that, for each ppf $s\in\Delta^q \cap
\C^{q-1}[-1,1]$ of degree $\le q+r$ with the only knot at
$0$ (\ie $s\in\Spl{q+r}{\{-1,0,1\}}$), and each partition $\bz{l}=
(z_i)_{i=0}^l$, $-1 = z_0 <z_1 < \dots <z_{l-1} < z_l = 1$ with
\[
\max_{0\leq i \leq {l-1}}(z_{i+1}-z_i)=:\delta<\delta(q,r) \,,
\]
 there exists a spline $\tilde
s\in\Delta^q \cap\Sspl{q+r}{\bz{l}}$ (\ie $\tilde s$ is a spline of minimal defect), such that
$\tilde s\equiv   s$ in  neighborhoods of $-1$ and   $1$,
and
\be\label{mainin}
\norm{s-\tilde s}{p}\le c(p,q,r) \omega_{q+r+1}(s,1)_p,
\ee
for all $0<p\le\infty$.
\end{lemma}

\begin{proof}
Since $s^{(q-1)}$ is   absolutely continuous  and differentiable
everywhere except for $0$ (because it is a continuous ppf of degree
$\leq r+1$ with the only breakpoint at $0$), we conclude that
$S:=s^{(q)}$ is a non-negative ppf of degree $\le r$ with the only
knot at $0$. Let     $p_1:=S|_{[-1,0)}$, $p_2:=S|_{(0,1]}$ be the
polynomial pieces of $S$, and denote
 $\eta:=\norm{p_1-p_2}{\infty}$ and
$d:=2r^2$.
Lemma~\ref{positive_piece} implies that there exists an interval
 $I$ (without loss of generality we may assume that it is a subset
of $[0,1]$) such that $|I|\ge (2d)^{-1}$ and
\begin{equation}\label{from_below}
S(x)\ge c_1 \eta, \quad\text{for all }x\in I\,,
\end{equation}
where $c_1$ depends on $r$ only.

If    $P_1:=s|_{[-1,0)}$,
$P_2:=s|_{(0,1]}$ are the original  polynomial pieces of $s$, then $p_1 = P_1^{(q)}$,
$p_2 = P_2^{(q)}$, and using \lem{markov} as well as \lem{klp} we have
\begin{eqnarray*}
\eta &= & \norm{P_1^{(q)}-P_2^{(q)}}{\infty} \leq (q+r)^{2q} \norm{P_1 -P_2 }{\infty} \\
&\leq& c(p,q,r)  \norm{P_1 -P_2 }{p} \leq c(p,q,r) \omega_{q+r+1}(s,1)_p \, .
\end{eqnarray*}

Now, let $I'$ be the subinterval of $I$ with the same right endpoint as $I$ and length
$|I'|= (5d)^{-1}$, so that
\[
I'\cap [-(5d)^{-1},(5d)^{-1}]=\emptyset.
\]
Suppose now that the index $\sti$ is such that $z_{\sti} \leq 0 < z_{\sti+1}$, and that
 $\delta <  (5d(d+2))^{-1}$. Then  each of the intervals
$[-(5d)^{-1},0]$ and $[0,(5d)^{-1}]$ contains at least $d+1$ points from $\bz{l}$:
\[
(z_i)_{i=\sti-d}^{\sti} \subset [-(5d)^{-1},0] \quad \mbox{\rm and} \quad
(z_i)_{i=\sti+1}^{\sti + d+1} \subset [0,(5d)^{-1}] \, .
\]
We can apply \lem{glue} to construct a non-negative spline $\TS$ of degree $r$ and highest smoothness ($\TS\in\C^{r-1}$)
having knots
$(z_i)_{i=\sti-d}^{\sti+d+1}$ only, coinciding with $S$ outside $\sJ := [z_{\sti-d}, z_{\sti + d+ 1}]$,
and such that
\[
\norm{S-\TS}{\infty}\le 2 \eta \, .
\]
We define  $\ints$ to be such that $\ints^{(q)} \equiv \TS$ and $\ints^{(\nu)}(-1) = s^{(\nu)}(-1)$ for
all $0\leq \nu\leq q-1$. Therefore,
\[
\ints(x):=\sum_{\nu=0}^{q-1}\frac{s^{(\nu)}(-1)}{\nu!}(x+1)^\nu + \frac{1}{(q-1)!} \int_{-1}^x (x-t)^{q-1}\TS(t)\,dt \,.
\]
The spline $\ints$ is almost what we need. Namely, it is in $\Delta^q \cap\Sspl{q+r}{\bz{l}}$, it coincides with $s$ in a neighborhood of $-1$,
and
since $S$ and $\TS$ may only differ on an interval $\sJ$ of length $|\sJ| \leq (2d+1)\delta$, we have
\begin{eqnarray*}
\norm{s-\ints}{\infty} &\leq& \norm{\frac{1}{(q-1)!} \int_{-1}^x (x-t)^{q-1}\left( S(t)-\TS(t)\right) \,dt}{\infty} \\
& \leq & \frac{2^q \eta}{(q-1)!} |\sJ| \leq 2^q (2d+1) \eta \delta =: c_2 \eta \delta\, .
\end{eqnarray*}
The only property that $\ints$ does not have is that it may not be coinciding with $s$ in a neighborhood of $1$.
To remedy this, we use \lem{beatson} and ``glue" $\ints$ to $s$ on the interval $I'$ preserving its properties.

Let $[a,b] := I'$ and recall that $b-a = (5d)^{-1}$. Now, let  $d_1:=2(q+r)^2$ and take $\delta < h:= \frac{b-a}{2(d_1+1)} = \frac{1}{10d(d_1+1)}$. Then, each interval $[a+jh, a+(j+1)h]$, $0\leq j\leq 2d_1+1$, contains at least one point $z_{i_j}$ from $\bz{l}$, and
we set $\bx{d_1} = (x_j)_{j=0}^{d_1}$, where $x_j := z_{i_{(2j)}}$, $0\leq j\leq d_1$. Note that
\[
\max_{0\leq j\leq d_1}(x_{j+1}-x_j) \le 3 h
\quad\text{and}\quad
\min_{0\leq j\leq d_1}(x_{j+1}-x_j) \ge h\, .
\]
Taking into account that both $s$ and $\ints$ are polynomials of degree $\leq q+r$ on $I'$, \lem{beatson} implies that there exists a spline $\sbe \in \Sspl{q+r}{\bz{l}}$ such that the only knots of $\sbe$ inside $I'$ are $x_j$, $0\leq j\leq d_1$,
$\sbe(x)$ is a number between $s(x)$ and $\ints(x)$ for $x\in [x_0, x_{d_1}] \subset I'$ and $\sbe$ coincides with $\ints$ on
$[-1, x_0]$ and with $s$ on $[x_{d_1}, 1]$.
Hence,
\[
\norm{s-\sbe}{\infty} \leq \norm{s-\ints}{\infty} \leq c_2  \eta \delta \, ,
\]
 (so that \ineq{mainin} immediately follows), and it only remains to verify that $\sbe \in \Delta^q[x_0, x_{d_1}]$.
To this end, it suffices to show that $\sbe^{(q)}(x) \geq 0$ for all $x\in J_j := [x_j, x_{j+1}]$, $0\leq j\leq d_1-1$.
 Let such an interval $J_j$ be fixed.
Using \lem{markov} and the fact that $\sbe -s$ is a polynomial of degree $\leq q+r$ on $J_j$ we have
\begin{eqnarray*}
 \norm{\sbe^{(q)} - s^{(q)}}{\C(J_j)}  &\leq&  \left( \frac{2(q+r)^2}{|J_j|} \right)^q \norm{s-\sbe}{\C(J_j)}
 \leq \left( \frac{2(q+r)^2}{h} \right)^q c_2 \eta \delta  \\
 & = & \left( 10 d d_1 (d_1+1)\right)^q c_2 \eta \delta =: c_3 \eta \delta \,,
 \end{eqnarray*}
 where $c_3$ depends only on $r$ and $q$.
 Therefore, for every $x\in J_j$, using \ineq{from_below} we have
 \[
 \sbe^{(q)}(x) \geq s^{(q)}(x) - c_3 \eta \delta \geq c_1 \eta - c_3 \eta \delta \geq 0
 \]
provided $\delta \leq c_1/c_3$. Combining the above restrictions on $\delta$ we see that it is possible to
take
\[
\delta(q,r) := \min\left\{  (10 d (d_1+1))^{-1}, c_1/c_3   \right\}\,,
\]
and the construction of $\sbe$ is complete.
\end{proof}

\begin{corollary} \label{newcor}
Let $q\in\N$, $r\in\N$. There there exists a constant
$\delta(q,r)>0$ such that, for each ppf 
 $s\in\Delta^q \cap \C^{q-1}[a,b]$ be
a ppf of degree $\le q+r$ with the only knot at $c \in (a,b)$ (\ie
$s\in\Spl{q+r}{\{a,c,b\}}$), and each
partition  $\bz{l}=
(z_i)_{i=0}^l$, $a = z_0 <z_1 < \dots <z_{l-1} < z_l = b$ with
\[
\max_{0\leq i \leq {l-1}}(z_{i+1}-z_i) < \delta(q,r) \min\{b-c, c-a\} \,,
\]
 there exists a spline $\tilde
s\in\Delta^q \cap\Sspl{q+r}{\bz{l}}$ (\ie $\tilde s$ is a spline of minimal defect), such that
$\tilde s\equiv   s$ in  neighborhoods of $a$ and   $b$,
and
\[
\norm{s-\tilde s}{\Lp[a,b]}\le c(p,q,r) \omega_{q+r+1}(s,[a,b])_p,
\]
for all $0<p\le\infty$.
\end{corollary}
\begin{proof}
Without loss of generality, assume $c\ge(b+a)/2$. We apply
Lemma~\ref{lemma3.6} to $s^*(x):=s(c+(b-c)x)$ and those knots
$(z_i-c)/(b-c)$ which belong to $[-1,1]$. Now, extending the
resulting ppf $\tilde s^*$ polynomially from $[-1,1]$ to all $\R$,
we can go back to the original interval $[a,b]$ using $\tilde s(t):=\tilde
s^*((t-c)/(b-c))$.
\end{proof}

\mbox{}

{\bf\em Proof of \thm{smoothing}}. It is sufficient to apply \cor{newcor} with $[a,b]
= \J_j$ and $c = z_j$, for each $1\leq j\leq n-1$. Construction of $\tilde s$ is now obvious. \endproof

\subsection{Convex and monotone spline smoothing: auxiliary results for the proof of \thm{c12}} \label{cms}

Recall that for a partition $\z:= (z_i)_{i=0}^n$, $a =: z_0 <z_1 < \dots <z_{n-1} < z_n := b$ of an interval $[a,b]$, $z_j :=a$, $j<0$, and
$z_j :=b$, $j>n$, and that $J_j=[z_j,z_{j+1}]$. The following lemma shows how a  piecewise polynomial convex function can be smoothed to be continuously differentiable without adding any extra knots and keeping the error small.

\begin{lemma}[Convex smoothing]\label{lemcms}
Let   $r\in\N_0$, $\z := (z_i)_{i=0}^n$, $a =: z_0 <z_1 < \dots <z_{n-1} < z_n := b$ be a partition of $[a,b]$, and let $s\in\Delta^2 \cap \Spl{r+2}{\z}$.
Then, there exists $\tilde s\in\Delta^2 \cap \Spl{r+2}{\z}\cap\C^1[a,b]$ such that, for any $1\leq j\leq n-1$ and all $0<p\le\infty$,
\be \label{con}
\norm{s-\tilde s}{\Lp[z_{j-1}, z_{j+1}]}\le c(p,r,\scale(\z)) \omega_{r+3}(s,[z_{j-2}, z_{j+2}])_p\, ,
\ee
and
\be\label{mon}
\norm{s'-\tilde s'}{\Lp[z_{j-1}, z_{j+1}]}\le c(p,r,\scale(\z)) \omega_{r+2}(s',[z_{j-2}, z_{j+2}])_p \, .
\ee
Moreover,
\be\label{last}
\tilde s^{(\nu)}(a) =s^{(\nu)}(a) \quad \mbox{\rm and}\quad \tilde s^{(\nu)}(b) =s^{(\nu)}(b)\,, \quad \nu=0,1\, .
\ee
\end{lemma}

{\em Proof.} This lemma is actually a simpler version of \cite[Lemma
1]{LeShe02} (see also \cite{She-one}). It was not proved in
\cite{LeShe02} for all $p>0$, and the construction of $\tilde s$ was
more involved there (because $s$ was allowed to change convexity).
For completeness, we recall this construction from \cite{LeShe02}
adopting it to our case, and showing how the estimates can be
obtained for all $p>0$.

With $p_i:=s|_{J_i}$, $0\leq i\leq n-1$,  denote
\[
a_i(x) :=
\frac{(z_{i+2}-z_{i+1})\left( p'_{i+1}(z_{i+1}) - p'_{i}(z_{i+1})\right) }{2(z_{i+1}-z_i)(z_{i+2}-z_i)} (x-z_i)^2\,,
\quad 0\leq i\leq n-2 \,,
\]
and
\[
b_i(x) :=
\frac{(z_{i}-z_{i-1})\left( p_{i}'(z_{i}) - p_{i-1}'(z_{i})\right) }{2(z_{i+1}-z_i)(z_{i+1}-z_{i-1})} (x-z_{i+1})^2\,,
\quad 1\leq i\leq n-1 \,,
\]
and also set $a_{n-1}(x) \equiv 0$ and $b_0(x) \equiv 0$. Then, for $x\in [z_i, z_{i+1}]$, $0\leq i\leq n-1$, we define
\[
\tilde s (x) := p_i(x) + a_i(x) + b_i(x) \, .
\]
It is now straightforward to verify that $s$ is a continuously differentiable convex function satisfying \ineq{last}, and it remains to prove \ineq{con} and \ineq{mon}.
For $1\leq j\leq n-1$, we have
\begin{eqnarray*}
\lefteqn{\norm{s-\tilde s}{\Lp[z_{j-1}, z_{j+1}]}
  \le  2^{1/p} \max_{i=j-1,j}  \norm{a_i + b_i}{\Lp[z_{i}, z_{i+1}]} } \\
& \leq &   c(p,r,\scale(\z)) |J_i|^{1+1/p} \max_{i=j-1,j, j+1\, ;\, 1\leq i\leq n-1 }  \left| p_{i}'(z_{i}) - p_{i-1}'(z_{i}) \right| \\
&\le& c(p,r,\scale(\z)) \omega_{r+3}(s,[z_{j-2}, z_{j+2}])_p \, ,
\end{eqnarray*}
where the last inequality follows from \lem{klp}. Similarly, \lem{klp} implies
\begin{eqnarray*}
\lefteqn{\norm{s'-\tilde s'}{\Lp[z_{j-1}, z_{j+1}]}
  \le  2^{1/p} \max_{i=j-1,j}  \norm{a_i' + b_i'}{\Lp[z_{i}, z_{i+1}]} } \\
& \leq &   c(p,r,\scale(\z)) |J_i|^{1/p} \max_{i=j-1,j, j+1\, ;\, 1\leq i\leq n-1 }  \left| p_{i}'(z_{i}) - p_{i-1}'(z_{i}) \right| \\
&\le& c(p,r,\scale(\z)) \omega_{r+2}(s',[z_{j-2}, z_{j+2}])_p \,.
\qquad \endproof
\end{eqnarray*}

\lem{lemcms} immediately implies the following result for monotone spline smoothing.

 \begin{corollary}[Monotone smoothing] \label{corms}
Let   $r\in\N_0$, $\z := (z_i)_{i=0}^n$, $a =: z_0 <z_1 < \dots <z_{n-1} < z_n := b$ be a partition of $[a,b]$, and let $s\in\Delta^1 \cap \Spl{r+1}{\z}$.
Then, there exists $\tilde s\in\Delta^1 \cap \Spl{r+1}{\z}\cap\C[a,b]$ such that, for any $1\leq j\leq n-1$ and $0<p\le\infty$,
\[
\norm{s-\tilde s}{\Lp[z_{j-1}, z_{j+1}]}\le c(p,r,\scale(\z)) \omega_{r+2}(s,[z_{j-2}, z_{j+2}])_p\, .
\]
Moreover,
\[
\tilde s(a) =s(a) \quad \mbox{\rm and}\quad \tilde s(b) =s(b)\, .
\]
\end{corollary}

In the case $n=2$, \ie when a ppf has only one breakpoint inside an interval $[a,b]$ we get the following corollaries for  convex and monotone spline smoothing.

\begin{corollary}  \label{cora}
Let   $r\in\N_0$, $\zz := (z_i)_{i=0}^2$, $a =: z_0 <z_1 < z_2 := b$ be a partition of $[a,b]$, and let $s\in\Delta^2 \cap \Spl{r+2}{\zz}$.
Then, there exists $\tilde s\in\Delta^2 \cap \Spl{r+2}{\zz}\cap\C^1[a,b]$ such that
\[
\norm{s-\tilde s}{\Lp[a, b]}\le c(p,r,\scale(\zz)) \omega_{r+3}(s,[a, b])_p
\]
for any $0<p\le\infty$. Moreover,
\[
\tilde s^{(\nu)}(a) =s^{(\nu)}(a) \quad \mbox{\rm and}\quad \tilde s^{(\nu)}(b) =s^{(\nu)}(b)\,, \quad \nu=0,1\, .
\]
\end{corollary}

\begin{corollary}  \label{mora}
Let   $r\in\N_0$, $\zz := (z_i)_{i=0}^2$, $a =: z_0 <z_1 < z_2 := b$ be a partition of $[a,b]$, and let $s\in\Delta^1 \cap \Spl{r+1}{\zz}$.
Then, there exists $\tilde s\in\Delta^1 \cap \Spl{r+1}{\zz}\cap\C[a,b]$ such that
\[
\norm{s-\tilde s}{\Lp[a, b]}\le c(p,r,\scale(\zz)) \omega_{r+2}(s,[a, b])_p
\]
for any $0<p\le\infty$. Moreover,
$\tilde s(a) =s(a)$ and $\tilde s(b) =s(b)$.
\end{corollary}

\mbox{}

{\bf\em Proof of \thm{c12}}. 
Let $\zzz_{3n}^\star$ be  a refinement of $\z$ obtained by adding two extra knots $l_i$ and $r_i$ in each interval
$J_i = [z_i, z_{i+1}]$, $0\leq i\leq n-1$, where $l_i := z_i + |J_i|/4$ and $r_i:= z_{i+1} - |J_i|/4$.
We now apply Corollaries~\ref{cora} and \ref{mora}
for every $1\leq i\leq n-1$, with 
$\zz = \left( r_{i-1}, z_i, l_i \right)$ to obtain $s^\star\in\Spl{q+r}{\zzz_{3n}^\star}\cap\Delta^q\cap\C^{q-1}[-1,1]$
satisfying
\[
\norm{s- s^\star}{\Lp[r_{i-1}, l_i]}
\le c(p,r) \omega_{q+r+1}(s,[r_{i-1}, l_i])_p\, ,\; 1\leq j\leq n-1\,.
\]
Moreover,   $s^\star \equiv s$ on $[l_i, r_i]$, $1\leq i\leq n-2$, and on $[-1, r_0]$ and $[l_{n-1}, 1]$,
and  therefore 
\be \label{tmpp}
\norm{s- s^\star}{\Lp(\J_j)}
\le c(p,r) \omega_{q+r+1}(s,\J_j)_p\, ,\; 0\leq j\leq n\,.
\ee
Now, using \thm{smoothing} with $s^\star$ and $\zzz_{3n}^\star$ instead of $s$ and $\z$, respectively, and 
taking into account that, if 
$\tbz{m}\in\refclass{\z}$, then $\tbz{m}\in\remesh{4\delta}{\zzz_{3n}^\star}$, we conclude that
\[
\norm{s^\star-\tilde s}{\Lp(\tilde\J_j)}
\le c(p,q,r) \omega_{q+r+1}(s^\star,\tilde\J_j)_p\, , \; 0\leq j\leq n  \,,
\]
where $\tilde\J_j$ stands for $[(l_{i-1}+r_{i-1})/2, (r_{i-1}+z_i)/2]$, 
$[(r_{i-1}+z_i)/2, (z_i+l_i)/2]$, or $[(z_i+l_i)/2, (l_i+r_i)/2]$,
and $l_i  = r_i := -1$ if $i<0$, and $l_i  = r_i := 1$ if $i\geq n$.
Finally, we notice that $(l_i+r_i)/2 = (z_i+z_{i+1})/2$ and, hence,
\[
\norm{s^\star-\tilde s}{\Lp(\J_j)}
\le c(p,q,r) \omega_{q+r+1}(s^\star,\J_j)_p\, , \; 0\leq j\leq n  \,,
\]
which together with \ineq{tmpp} completes the proof of the theorem.
\endproof

\end{document}